%% file: dmds_final.tex
\documentclass[11pt]{amsart}
\usepackage{comment}
\usepackage{amssymb}
\usepackage{graphicx}
\usepackage{subcaption}
\input MM_MathDefs

\articletheorems
\newtheorem{lm}[thm]{Lemma}
\newtheorem{df}[thm]{Definition}
\newfont{\mbf}{msbm10 scaled 1100}
\newfont{\mmbf}{msbm10 scaled 800}

\def\Z{\mbox{\mbf Z}}

\def\N{\mbox{\mbf N}}

\def\z{\mbox{\mmbf Z}}

\def\for{\hspace{1ex} \mbox{ for }}

\newcommand{\ac}{\operatorname{ac}}
\newcommand{\Sol}{\operatorname{Sol}}

\usepackage[cp1250]{inputenc}

\begin{document}
\title[]{Weak index pairs and the Conley index for discrete multivalued dynamical systems}
\author{Bogdan Batko and Marian Mrozek}
\date{}
\address{Bogdan Batko\\
Division of Computational Mathematics,
Faculty of Mathematics and Computer Science,
Jagiellonian University,
ul. St. Łojasiewicza 6,
30-348 Krak\'ow,
Poland
}
\email{bogdan.batko@ii.uj.edu.pl}
\address{Marian Mrozek\\
Division of Computational Mathematics,
Faculty of Mathematics and Computer Science,
Jagiellonian University,
ul. St. Łojasiewicza 6,
30-348 Krak\'ow,
Poland
}
\email{marian.mrozek@ii.uj.edu.pl}
\date{}
\thanks{
   This research is partially supported by the EU under the TOPOSYS project FP7-ICT-318493-STREP
   and by the Polish National Science Center under Ma\-estro Grant No. 2014/14/A/ST1/00453.
}

\begin{abstract}
Motivated by the problem of reconstructing dynamics from samples we revisit the Conley index theory for discrete multivalued dynamical systems \cite{KM95}. We introduce a new, less restrictive definition of the isolating neighbourhood. It turns out that then the main tool for the construction of the index, i.e. the index pair, is no longer useful. In order to overcome this obstacle we use the concept of weak index pairs.
\end{abstract}
\maketitle
\section{Introduction}
Multivalued dynamical systems are dynamical systems without forward uniqueness of solutions.
The interest in such systems originated in the qualitative analysis of differential equations without uniqueness of solutions
and differential inclusions \cite{AC84}. Surprisingly, the theory of multivalued dynamics is also important in the study
of single valued dynamical systems, particularly in the rigourous numerical analysis of differential
equations and iterates of maps. All what we can rigorously extract from a finite numerical experiment
is a multivalued map enclosing as a selector
the generator of the discrete dynamical system or time $t$ map of a differential equation.
Depending on the quality of the enclosure, some qualitative features of the single valued dynamics may be rigorously
inferred from the multivalued dynamics via topological invariants such as the Conley index. This type of analysis was
originated in \cite{MiMr95} and since then applied to many concrete problems.

Recall that according to \cite{KM95} a compact set $N$ is an isolating neighbourhood of a multivalued map $F$ if
\begin{equation}
\label{eq:mv-s-iso}
\dist(\Inv N,\bd N)> \max\{\diam F(x)\mid x\in N\},
\end{equation}
where $\Inv N$ stands for the invariant part of $F$ in $N$ (see Definition~\ref{def:inv}).
A compact set $S$ is an isolated invariant set if there exists an isolating neighbourhood $N$ such that $S=\Inv N$.
A slightly weaker but essentially similar definition of an isolated invariant set is presented in \cite{S06}.
The aim of this paper is to define the Conley index of a multivalued map $F$ for an isolating neighbourhood defined as
a compact set $N$ satisfying the condition
\begin{equation}
\label{eq:mv-iso}
\Inv N\subset\Int N,
\end{equation}
that is the same condition as for single-valued maps.
To avoid ambiguity, in this paper we refer to
an isolating neighbourhood in the sense of \eqref{eq:mv-s-iso}
as a {\em strongly isolating neighbourhood} and by an {\em isolating neighborhood} we mean
a compact set $N$ satisfying \eqref{eq:mv-iso}. We extend this convention to isolated invariant sets.
Obviously, every strongly isolating neighborhood is an isolating neighborhoood
and every strongly isolated invariant set is an invariant set.
We show (see Example~\ref{ex1}) that isolating neighbourhoods in the sense of \eqref{eq:mv-iso} may not admit
index pairs, the main tool for the construction of the Conley index.
However, we prove that weak index pairs, a concept adapted form \cite{M06}, do exist (see Theorem~\ref{thm:existence}).
Moreover, after suitable modifications the construction of the Conley index in \cite{KM95} works also
with weak index pairs (see Theorem~\ref{th_ind}) and the new definition generalizes the earlier definitions
(Theorem~\ref{thm:gen}).

The motivation for the proposed generalization comes from sampled dynamics.
More precisely, the Conley index for isolating neighbourhoods of multivalued maps
in the sense of \eqref{eq:mv-iso} may be useful in the reconstruction
of the qualitative features of an unknown dynamical system on the basis of the available experimental data only.
In general, the reconstruction problem is difficult. The potential benefits from applying topological
tools to the reconstruction problem are demonstrated in \cite{MiMrReSz99,MiMrReSz99b} and recently also in \cite{EdJaMr15}.
We use the leading example of \cite{EdJaMr15} to explain in Section~\ref{sec:m-ex} the benefits of
the proposed generalization for sampled dynamical systems.
The theory presented in this paper is also needed in extending towards the Conley index theory
the results of  \cite{KMW14}. These recent results provide
the formal ties between the classical dynamics and the combinatorial dynamics in the sense of Forman \cite{Fo98b}.

The organization of the paper is as follows.
A motivating example is presented in Section 2.
Section 3 presents preliminaries needed in the paper. Section 4 introduces the new definition of the isolating neighbourhood
and discusses its relation with the former definitions. We show that unlike strongly isolating neighbourhoods, isolating neighbourhoods do not guarantee the existence of index pairs, however, they admit weak index pairs. In Section 5 we present several properties of weak index pairs that we need to well pose the definition of the Conley index in Section 6.
In Section 7 we show that an isolated invariant set may be treated as a strongly isolated invariant set, if we embed it into an appropriate "dual" space.
In the last section we compute the Conley indexes of the two examples studied in Section 2.
\section{Sampling dynamics: a motivating example}
\label{sec:m-ex}

There are many dynamical systems for which neither  analytic computations nor a rigourous numerical analysis are possible.
Several issues cause this situation. The total lack of a mathematical model,
inadequate knowledge of parameters or complicated global nonlinearities may prevent computations at all.
Additionally, sensitive dependence on initial conditions, blowup of error estimates or lack of sufficient computational power
may make the rigourous numerical analysis infeasible.
Thus, sampling the dynamical system is often the only way to infer some knowledge about the system.
By sampling we mean collecting a finite amount of points and their approximate images
under the generator of the dynamical system. This may be done in a physical experiment or in a numerical experiment
if sufficient information about the system is available.
Then, there is the question whether this finite amount of data is sufficient to obtain some global, general knowledge about
the system and how to do it.
In general, this is difficult, particularly
in the case of chaotic dynamics, when the intrinsic problems of chaotic systems
are amplified by the introduction of noise, parameter drift, and experimental error.
In this situation the coarseness of topological invariants as the Conley index and persistent homology
 \cite{EdLeZo02} turn out to be helpful.
In particular, it is demonstrated in \cite{MiMrReSz99,MiMrReSz99b} that the Conley index combined with multivalued approach
suffices to detect chaotic dynamics in experimental data. And it is shown in a recent paper \cite{EdJaMr15}
that the eigenvalues of the map induced in homology may be reconstructed from a very small
sample of a continuous map by means of a technique known as persistence.
Both papers are concerned with a similar situation which may be roughly described as follows.

Assume $X\subset \RR^d$ is an unknown space,  $f:X\to X$ is the generator of an unknown dynamical system,
$A\subset X$ is a collection of finite samples of $X$ and $g:=f_{|A}:A\to X$ samples the map $f$.
Note that in a more realistic situation the set of samples $A$ may only lie in a vicinity of $X$ and
the sampling map $g$ may only have its graph nearby $f$. Also, as in \cite{MiMrReSz99,MiMrReSz99b}, samples may not
be available directly but only through a set of measurements. The reconstruction may take the form of
a simplicial map as in \cite{EdJaMr15} or, as in \cite{MiMrReSz99,MiMrReSz99b},
may be a multivalued map constructed as follows.
Divide $\RR^d$ into a grid of compact, acyclic sets, for instance hypercubes
(see \cite{M99} for a general definition of a grid).
Assume that for a finite family $\cA$ of grid elements we may construct
a possibly small acyclic set $\ac\cA\subset\RR^d$ such that $\bigcup\cA\subset\ac\cA$.
In the case of a cubical grid this may be the smallest hypercuboid containing all the hypercubes
in the family $\cA$.
Let $\cK$ denote the family of grid elements whose
intersection with $A$ is nonempty and assume the grid elements and/or the sample set $A$
are large enough to guarantee that $X\subset\bigcup\cK$.
Define the combinatorial multivalued map
\[
    \cG:\cK\ni Q\mapsto\setof{P\in\cK\mid \exists x\in Q:\;g(x)\in P}\subset\cK.
\]
and the multivalued map
\begin{equation}
\label{eq:F-from-cG}
    F:X\ni x\mapsto \ac\left(\bigcup_{x\in Q\in\cK}\cG(Q)\right)\subset\RR^d.
\end{equation}

It is not difficult to prove that $F$ is upper semicontinuous and $g(x)\in F(x)$ for any $x\in A$.
However, we cannot expect that the whole $f$ is a selector for $F$, that is
$f(x)\in F(x)$ for any $x\in X$.
Even worse, we cannot expect that $F$ has any continuous selector at all.
In practice, the continuous selector requirement often fails due to insufficient number of samples and a locally expanding
behaviour of $f$. But, in the method proposed in \cite{MiMrReSz99,MiMrReSz99b}
the continuous selector requirement has to be satisfied, because the approach is based on the Conley
index for continuous maps and not on the Conley index for upper-semicontinuous multivalued maps.
The multivalued map $F$ is used in  \cite{MiMrReSz99,MiMrReSz99b} only to construct the so called index pair.
An index pair is needed in the computation of the Conley index for the selector.
As a remedy, one can enlarge the values of $F$ to guarantee the existence of a continuous selector.
But, such an enlarging often leads to overestimation and loss of isolation properties.
An alternative is to apply the Conley theory directly to the multivalued map $F$ constructed from the experimental data
and extend the results to the unknown generator $f$ by means of continuation. The theory developed in \cite{KM95,S06}
only requires that $F$ is upper-semicontinuous but does not require the existence of a continuous selector.
In particular, the generator $f$ may be located nearby $F$ but not necessarily inside $F$.
Unfortunately, such an approach suffers from the very restrictive nature of condition \eqref{eq:mv-s-iso}
in the definition of isolating neighbourhood adopted in \cite{KM95}.
In practice, it is difficult to satisfy condition \eqref{eq:mv-s-iso},
because controlling the size of values of $F$ is either very expensive
or just not possible.

To illustrate the problem, let us go back to the motivating example of \cite{EdJaMr15}.
Take $S^1=\RR/\ZZ$ and, to keep the notation simple,
identify a real number $x\in\RR$ with its equivalence class $[x]\in \RR/\ZZ$.
Consider the self-map
\begin{equation}
\label{eq:f}
   f: S^1\ni x\mapsto 2x\in S^1.
\end{equation}
Let $x_i:=\frac{i}{16}$ and take $A:=\setof{x_i\mid  i=0,1,2,\dots,15}\subset S^1$ as the finite sample.
Since $f(A)\subset A$, the restriction $g:=f_{|A}$ is an exact sample of $f$ on $A$.
Consider the grid $\cA$ on $S^1$ consisting of intervals $[x_i-\frac{1}{32},x_i+\frac{1}{32}]$.
The graph of the multivalued map $F$ obtained via \eqref{eq:F-from-cG} from the smallest combinatorial enclosure
$\cG$ of $g$ on $\cA$ is presented in Figure~\ref{fig:z2on16samples}.
Note that $F$ does not admit a continuous selector.

\begin{ex}\label{e1}
{\em
  Note that $0$ is a hyperbolic fixed point of $f$. Thus, $\{0\}$ is an isolated invariant set of $f$.
  It belongs to $S:=[\frac{31}{32},\frac{1}{32}]\in\cA$. It is straightforward to observe that $S$ is an invariant
  set for $F$. Let $N:=[\frac{15}{16},\frac{1}{16}]$. Since $F(S)=[\frac{27}{32},\frac{5}{32}]$, the set $N$ is not
  a strongly isolating neighbourhood for $F$ and $S$. However, one easily verifies that $N$ is an isolating neighbourhood
  for $F$ and $S$. Also, it is not difficult to verify that $N$ fails to be an isolating neighborhood if the values
  of $F$ are enlarged to make $f$ a continuous selector of $F$.
  We will show in Section \ref{sec:ex} that the Conley index of $S$ for $F$ is the same as the Conley index
  of $\{0\}$ for $f$.
}
\end{ex}

\begin{ex}\label{e2}
{\em
Note that $\{\frac{1}{3},\frac{2}{3}\}$ is a hyperbolic periodic trajectory of $f$. In particular, it is
an isolated invariant set for $f$. Consider the cover of this set by elements of the grid $\cA$ and set
$S:=[\frac{9}{32},\frac{13}{32}]\cup [\frac{19}{32},\frac{23}{32}]$. It is easy to see that $S$ is an invariant set for $F$,
because each point of $S$ belongs to a $2$-periodic trajectory of $F$ in $S$.  We have $F(S)=[\frac{3}{32},\frac{29}{32}]$.
Thus, if $S$ is a strongly isolated invariant set, then any strongly isolating neighbourhood isolating $S$ must contain
$F(S)$. But, $\{\frac{7}{32},\frac{11}{32},\frac{25}{32},\frac{21}{32}\}$ is a $4$-periodic trajectory of $F$ contained
in $F(S)$. Thus, any strongly isolating neighbourhood containing $S$ has its invariant part essentially larger than $S$.
It follows that $S$ is not a strongly isolated invariant set for $F$. However, it is easy to see that
$N:=[\frac{17}{64},\frac{27}{64}]\cup [\frac{37}{64},\frac{47}{64}]$ is an isolating neighbourhood isolating precisely $S$.
Thus, the grid consisting of $16$ equal intervals suffices to pickup $S$ as an isolated invariant but not as a strongly isolated invariant set. We will show in Section \ref{sec:ex} that the Conley index of $S$ for $F$ is the same as the Conley index
of $\{\frac{1}{3},\frac{2}{3}\}$ for $f$.
}
\end{ex}

\begin{figure}[ht]
\begin{center}
\includegraphics[width=0.9\textwidth]{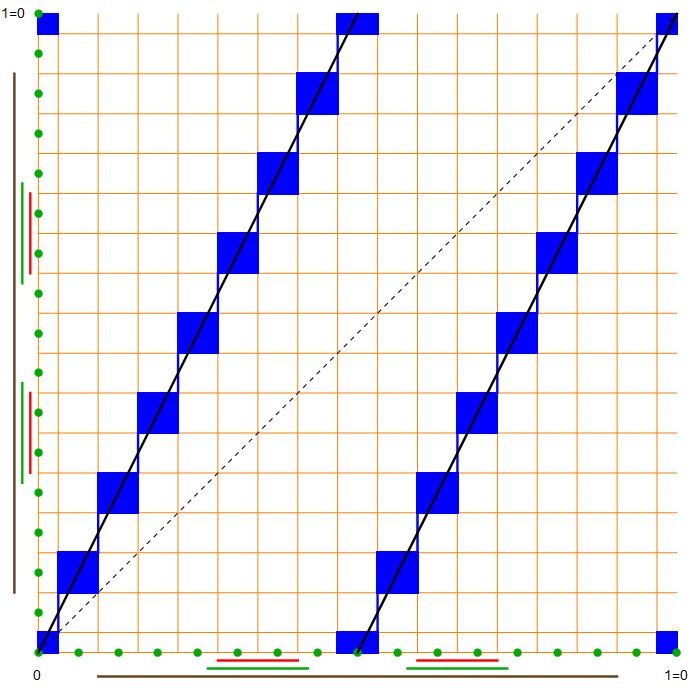}
\caption{The graph of the map $f$ given by \eqref{eq:f}, marked in black, and its sampling.
The $16$ sampled points are marked with green dots. The grid consisting of 16 intervals is marked in orange.
The graph of $F$ constructed from the sampling points via \eqref{eq:F-from-cG} is marked in blue.
A candidate for an isolated invariant set $S$ is marked in red. Its image $F(S)$, showing that $S$ is not a strongly
isolated invariant set, is marked in brown.
An isolating neighbourhood $N$ for $S$ is marked in green.
}
\label{fig:z2on16samples}
\end{center}
\end{figure}

\section{Preliminaries}
We denote the sets of all integers, non-negative integers and non-positive integers by $\Z$, $\Z^+$ and $\Z^-$, respectively. By an interval in $\Z$ we mean a trace of a closed real interval in $\Z$.

Given a topological space $X$ and $A\subset X$, the notation $\cl_XA$, $\Int _XA$ and $\bd _XA$ will be used
respectively for the closure, the interior and the boundary of $A$ in $X$. If the space is clear from the context, we shall drop the subscript $X$ in this notation.

Let $\mathcal{P}(Y)$ stand for the set of all subsets of a given topological space $Y$. A mapping $F:X\to \mathcal{P}(Y)$ is called {\em upper semicontinuous} ({\em usc} for short) if for an arbitrary closed set $B\subset Y$ its {\em large counter image} under $F$, i.e. the set
$$
F^{-1}(B):=\{x\in X\mid  F(x)\cap B \neq\emptyset\},
$$
is closed. It is equivalent to the assertion that the set
$\{x\in X\mid  F(x)\subset B\}
$, called the {\em small counter image} of $B$,
is open for any open $B\subset Y$. Recall that any usc mapping with compact values has a closed graph and it sends compact sets into compact sets. If $F:X\to \mathcal{P}(Y)$ is usc then its {\em effective domain}, i.e. the set $D(F):=\{x\in X\mid  F(x)\neq\emptyset\}$, is closed.

Given $A\subset X$ we define its {\em image} under $F$ as $F(A):=\bigcup\{F(x)\mid x\in A\}$. If $F:X\to\mathcal{P}(Y)$ and $G:Y\to\mathcal{P}(Z)$ then the {\em composition} $G\circ F:X\to\mathcal{P}(Z)$ is defined by
$$
(G\circ F)(x):=\bigcup\{G(y)\mid y\in F(x)\}\for x\in X.
$$
Finally, if $F:X\to\mathcal{P}(X)$ then by $F^k$, for $k\in\Z^+\setminus\{0\}$, we understand the  composition
of $k$ copies of $F$.
\begin{df}
{\rm (cf. \cite[Definition 2.1]{KM95}). An usc mapping $F:X\times\Z\to\mathcal{P}(X)$ with compact values is called a {\em discrete multivalued dynamical system} ({\em dmds}) if the following conditions are satisfied:
\begin{itemize}
\item[(i)] for all $x\in X$, $F(x,0)=\{x\}$,
\item[(ii)] for all $n,m\in\Z$ with $nm\geq 0$ and all $x\in X$, $F(F(x,n),m)=F(x,n+m)$,
\item[(iii)] for all $x,y\in X$, $y\in F(x,-1)\Leftrightarrow x\in F(y,1)$.
\end{itemize}
}
\end{df}
Note that $F^n(x):=F(x,n)$ coincides with a superposition of $F^1:X\to\mathcal{P}(X)$ or its inverse $(F^1)^{-1}$. Therefore we will call $F^1$ the {\em generator} of the dmds $F$. Moreover, we will denote the generator simply by  $F$ and identify it with the dmds.

In what follows we assume that $F$ is a given dmds.

\begin{df}
{\rm (cf. \cite[Definition 2.3]{KM95}). Let $I\subset\Z$ be an interval containing $0$. A single valued mapping $\sigma:I\to X$ is called a {\em solution for $F$ through $x\in X$} if $\sigma (0)=x$ and $\sigma (n+1)\in F(\sigma(n))$ for all $n,n+1\in I$.
}
\end{df}
\begin{df}
\label{def:inv}
{\rm Given $N\subset X$ we define the following sets
\begin{eqnarray*}
\Inv^+N&:=&\{x\in N\mid \exists\,\sigma:\Z^+\to N\mbox{ a solution  for }F\mbox{ through }x\},\\
\Inv^-N&:=&\{x\in N\mid \exists\,\sigma:\Z^-\to N\mbox{ a solution  for }F\mbox{ through }x\},\\
\Inv N&:=&\{x\in N\mid  \exists\,\sigma:\Z\to N\mbox{ a solution  for }F\mbox{ through }x\},
\end{eqnarray*}
called the {\em positive invariant part, negative invariant part} and the {\em invariant part of} $N$, respectively.
}
\end{df}
 Note that, by (i), $\Inv N=\Inv ^+N\cap \Inv ^-N$.

We will frequently consider pairs of topological spaces. For the sake of simplicity we will denote such pairs by single capital letters and then the first or the second element of the pair will be denoted by adding to the letter the subscript $1$ or $2$, respectively. In other words, if $P$ is a pair of spaces then $P=(P_1,P_2)$ where $P_1$, $P_2$ are topological spaces. Consequently, the rule extends to any relation $R$ between pairs $P$ and $Q$, i.e. any statement that pairs $P$ and $Q$ are in a relation $R$ will mean that $P_i$ is in a relation $R$ with $Q_i$ for $i=1,2$. According to our general assumption, whenever we say that $F$ is a map of pairs $P$ and $Q$ it means that $F$ maps $P_i$ into $Q_i$ for $i=1,2$.
\section{Definition of an isolating neighbourhood and construction of a weak index pair}
Let $F:X\to\mathcal{P}(X)$ be a given dmds. We begin this section with the definition of an isolating neighbourhood.
\begin{df}\label{in}
{\rm A compact subset $N\subset X$ is an {\em isolating neighbourhood} for $F$ if $\Inv N\subset \Int N$.
}
\end{df}
Note that the above definition generalizes earlier notions of isolating neighbourhoods for multivalued maps.
Recall that in \cite{KM95}
an isolating neighbourhood $N$ in a locally compact metric space is defined as a compact set satisfying
\begin{equation}\label{inKM}
\dist (\Inv N,\bd N)>\max \{\diam F(x)\mid x\in N\}.
\end{equation}
In \cite{S06} the isolation property has been slightly relaxed i.e. instead of (\ref{inKM}) it is required that
\begin{equation}\label{inKS}
\Inv N\cup F(\Inv N)\subset\Int N.\end{equation}
Note that if $F$ is single valued, conditions (\ref{inKM}), (\ref{inKS}) and the assertion of Definition \ref{in}, coincide. However, as we are going to show below, for a general mv map $F$ this is not the case.
In order to avoid the misunderstanding, throughout this paper isolating neighbourhoods in the sense of \cite{KM95} or \cite{S06} will be called {\em strongly isolating neighbourhoods}.
\begin{prop}\label{prop_strong_in}
Any strongly isolating neighbourhood is an isolating neighbourhood. The converse is not true.
\end{prop}
\begin{proof} The first statement is obvious. The second follows from Example \ref{e1}.
\qed\end{proof}
\begin{df}
\label{def:iis}
{\rm We say that a compact set $S\subset X$ is {\em invariant} with respect to $F$ if $S=\Inv S$. It is called an {\em isolated invariant set} if it admits an isolating neighbourhood $N$ for $F$ such that $S=\Int N$. If in the above assertion $N$ is a strongly isolating neighbourhood then $S$ will be called a {\em strongly isolated invariant set}.
}
\end{df}
As a straightforward consequence of Proposition \ref{prop_strong_in} and Example \ref{e2} we have the following.
\begin{prop}
\label{prop:siis-iis}
Any strongly isolated invariant set is an isolated invariant set. The converse is not true.
\end{prop}

The main tool in constructing the Conley index for flows as well as for discrete dynamical systems is the index pair.
\begin{df}\label{ip}
{\rm
A pair $P=(P_1,P_2)$ of compact sets $P_2\subset P_1\subset N$ is called an {\em index pair} in $N$ if
\begin{itemize}
\item[(a)] $F(P_i)\cap N\subset P_i$ for $i\in\{1,2\}$,
\item[(b')] $F(P_1\setminus P_2)\subset N$,
\item[(c)] $\Inv N\subset\Int (P_1\setminus P_2)$.
\end{itemize}
}
\end{df}
An index pair appears to be an effective tool also in studying discrete multivalued dynamical systems, when strongly isolating neighbourhoods are considered. However, as Example \ref{ex1} shows, invariant sets isolated in the sense of Definition \ref{in} do not necessarily guarantee the existence of index pairs.
\begin{ex}\label{ex1}
{\rm Consider the dmds $F$, the isolated invariant set $S$ and its isolating neighbourhood $N$, defined in Example \ref{e2}  (see Figure \ref{fig:z2on16samples}).
Suppose that there exists an index pair $P=(P_1,P_2)$ in $N$. Then, by (b') and (c), $[\frac{3}{32},\frac{29}{32}]=F(S)\subset N=[\frac{17}{64},\frac{27}{64}]\cup [\frac{37}{64},\frac{47}{64}]$, a contradiction.
}
\end{ex}
To avoid this obstacle, in what follows we adapt to our needs the notion of a weak index pair introduced in \cite{M06}.
We will prove that for any isolated invariant set a weak index pair exists, and that this is sufficient to well pose the definition of the Conley index.

We define an $F$-{\em boundary} of a given set $A\subset X$ by
$$
\bd _F(A):=\cl A\cap\cl (F(A)\setminus A).
$$
\begin{df}\label{wip}
{\rm A pair $P=(P_1,P_2)$ of compact sets $P_2\subset P_1\subset N$ is called a {\em weak index pair} in $N$ if
\begin{itemize}
\item[(a)] $F(P_i)\cap N\subset P_i$ for $i\in\{1,2\}$,
\item[(b)] $\bd _FP_1\subset P_2$,
\item[(c)] $\Inv N\subset\Int (P_1\setminus P_2)$,
\item[(d)] $P_1\setminus P_2\subset\Int N$.
\end{itemize}
}
\end{df}
The following straightforward proposition shows that, apart from condition (d), the axioms of a weak index pair are less restrictive than those of an index pair.
\begin{prop}
\label{prop:ip-wip}
Any index pair $P$ in $N$ such that $P_1\setminus P_2\subset\Int N$ is a weak index pair in $N$.
\end{prop}
\begin{proof}
We shall verify that $P$ satisfies condition (b). Suppose the contrary and take a $y\in\bd _F(P_1)\setminus P_2$.
Then, there exists a sequence $\{y_n\}\subset F(P_1)\setminus P_1$ convergent to $y$.
Since $y_n\in F(P_1)$, by (a) we have $y_n\notin N$. However, $y\in P_1\setminus P_2\subset\Int N$,
hence $y_n\in N$ for large enough $n$, a contradiction.
\qed\end{proof}
Given $N\subset X$ and an interval $I$ in $\ZZ$ set
\[
   \Sol(N,F,I):=\setof{\sigma:I\to N\text{ a solution  for $F$}}
\]
and for $x\in N$ and $n\in\Z^+$ define
\begin{eqnarray*}
F_{N,n}(x)&:=& \{y\in N\mid \exists\,\sigma\in\Sol(N,F,[0,n]):\,\sigma(0)=x,\sigma(n)=y\},\\
F_{N,-n}(x)&:=&\{y\in N\mid \exists\,\sigma\in\Sol(N,F,[-n,0]):\,\sigma(-n)=y,\sigma(0)=x\}, \\
F_N^+(x)&:=&\bigcup_{n\in\z^+}F_{N,n}(x),\\
F_N ^-(x)&:=&\bigcup_{n\in\z^+}F_{N,-n}(x).
\end{eqnarray*}

\begin{lm}\label{lem_8}
Let $N\subset X$ be compact and suppose that $D(F_{N,n})\neq\emptyset$ for all $n\in\Z^+$. Then $\Inv N\neq\emptyset$. Moreover, $\Inv ^{(\pm )}N=\bigcap\{D(F_{N,n})\mid n\in\Z \mbox{ }(n\in\Z^+\mbox{, }n\in\Z^-\mbox{, respectively})\}$.
\end{lm}
\begin{proof}
We start with the proof for $\Inv ^+N$. By \cite[Proposition 2.7]{KM95}, for each $n\in\Z$, the map
$F_{N,n}$ is usc. This implies that $D(F_{N,n})$ is compact. Thus, the sequence $\{D(F_{N,n})\}_{n\in\z^+}$ has a nonempty intersection $K$ as a decreasing sequence of nonempty compact sets. We shall prove that $\Inv ^+N=K$. Clearly, $\Inv ^+N\subset K$. Suppose that $x\in K$. Then, for each $n\in\Z^+$ there exists a solution $\sigma_n:[0,n]\to N$ for $F$ through $x$. We construct a solution $\sigma:\Z^+\to N$ for $F$ through $x$ by induction. Evidently, $\sigma (0)=x$. Suppose that $\sigma_{|[0,n]}$ is constructed and that there is a sequence $\{k_i\}$ with $k_i>n$ such that $\sigma _{k_i}(n)\to\sigma (n)$ as $i\to\infty$. Since $N$ is compact, passing to the subsequence, if necessary, we may assume that $\sigma _{k_i}(n+1)\to y (n+1)\in N$ as $i\to\infty$. By the closed graph property, $y_{n+1}\in F(\sigma(n))$, hence we may set $\sigma(n+1):=y_{n+1}$.

Now we focus on $\Inv ^-N$. Observe that if $D(F_{N,n})\neq\emptyset$ for all $n\in\Z^+$ then also $D(F_{N,n})\neq\emptyset$ for all $n\in\Z^-$. Indeed, if $\sigma:[0,n]\to N$ is a solution for $F$ through $\sigma(0)$ then $\sigma ':[-n,0]\ni i\mapsto \sigma (i+n)\in N$ is a solution for $F$ through $\sigma ' (0)$. The remaining part of the proof for $\Inv ^-N$ runs similarly as for $\Inv ^+N$.

Since $\Inv N=\Inv ^+N\cap\Inv ^-N$, we infer that $\Inv N=\bigcap\{D(F_{N,n})\mid n\in\Z\}$. By the equality $\bigcap\{D(F_{N,n})\mid n\in\Z\}=\bigcap\{\bigcap_{n=-k} ^kD(F_{N,n})\mid k\in\Z^+\}$, the set $\Inv N$ is nonempty as a decreasing sequence of nonempty compact sets.
\qed\end{proof}
For the reader's convenience let us quote the following two lemmas from \cite{KM95}.
\begin{lm}\label{l_2.9}
{\rm (cf. \cite[Lemma 2.9]{KM95})} Let $N\subset X$ be compact. Then
\begin{itemize}
\item[(i)] The sets $\Inv ^+N$, $\Inv ^-N$ and $\Inv N$ are compact;
 \item[(ii)] If $A$ is compact with $\Inv ^-N\subset A\subset N$ then $F_N ^+ (A)$ is compact.
\end{itemize}
\end{lm}
\begin{lm}\label{l_2.10}{\rm
(cf. \cite[Lemma 2.10]{KM95})} Let $K\subset N$ be compact subsets of $X$ such that $K\cap\Inv ^+N=\emptyset$ {\rm (}respectively $K\cap\Inv ^-N=\emptyset${\rm )}. Then
\begin{itemize}
\item[(i)] $F_{N,n}(K)=\emptyset$ for all but finitely many $n>0$ {\rm (}respectively $n<0${\rm )};
 \item[(ii)] The mapping $F_N ^+$ (respectively $F_N ^-${\rm )} is usc on $K$;
 \item[(iii)] $F_N ^+ (K)\cap\Inv ^+N=\emptyset$ {\rm (}respectively $F_N ^- (K)\cap\Inv ^-N=\emptyset${\rm )}.
\end{itemize}
\end{lm}
The main result of this section reads as follows.
\begin{thm}
\label{thm:existence}
Let $N$ be an isolating neighbourhood for $F$. For every neighbourhood $W$  of $\Inv N$ there exists a weak index pair $P$
in $N$ such that $P_1\setminus P_2\subset W$.
\end{thm}
\begin{proof}
Let $U,V$ be open in $N$ such that $\Inv ^+N\subset U$, $\Inv ^-N\subset V$ and $U\cap V\subset W\cap\Int N$. By \cite[Lemma 2.11]{KM95} there exists a compact neighbourhood $A$ of $\Inv ^-N$ such that $F_N ^+(A)\subset V$. We define
$$
P_1:=F_N ^+(A),\; P_2:=F_N ^+(P_1\setminus U).
$$
The proof that $P_1,P_2$ are compact, $P_1\setminus P_2\subset W$, $P_2\subset P_1$ and  (a), (c) are satisfied is analogous as in \cite[Theorem 2.6]{KM95}. We present it here for the sake of completeness. Since $P_1\subset V$ and $P_1\setminus U\subset P_2$, we have $P_1\setminus P_2\subset U$ and consequently $P_1\setminus P_2\subset U\cap V\subset W$.
$P_1$ is compact by Lemma \ref{l_2.9} (ii). Since $P_1\setminus U$ is compact, by Lemma \ref{l_2.10} (ii) so is $P_2$. Clearly, $P_2\subset F_N ^+(A)\subset P_1$.

To verify (a) take $x\in P_i$ and $y\in F(x)\cap N$. Then, there exists a solution $\sigma:[0,n]\to N$ with $\sigma (n)=x$ and $\sigma (0)\in A$ in the case $i=1$ or $\sigma (0)\in P_1\setminus U$ in the case $i=2$. One may extend $\sigma$ to $[0,n+1]$ by putting $\sigma (n+1):=y$. Hence $y\in P_i$.

In order to verify (c) observe that $P_1$ is a neighbourhood of $\Inv ^-N$ and, by Lemma \ref{l_2.10} (iii),
the set $N\setminus P_2$ is a neighbourhood of $\Inv ^+N$. Therefore, $P_1\setminus P_2=P_1\cap (N\setminus P_2)$ is a neighbourhood of $\Inv ^-N\cap\Inv ^-N=\Inv N$.

We still need to prove (b) and (d).
Since $P_1\setminus P_2\subset W$ and $W \subset\Int N$, condition (d) is verified.
In order to prove (b) ssume to contrary that there exists an $x\in\bd _FP_1\setminus P_2$. Then $x\in P_1\cap U\subset V\cap U\subset \Int N$. Since $x\in \bd _FP_1$, there exists an $\bar{x}\in\Int N\cap F(P_1)\setminus P_1$. Let $u\in P_1$ be such that $\bar{x}\in F(u)$ and let $\sigma :[0,k]\to N$ be a solution such that $\sigma(0)\in A$ and $\sigma (k)=u$. Since $\bar{x}\in N\cap F(u)$, we may extend $\sigma $ to $[0,k+1]$ by setting $\sigma (k+1):=\bar{x}$. This shows that $\bar{x}\in F_N ^+(A)=P_1$, a contradiction.
\qed\end{proof}
\section{Properties of weak index pairs}
In this section we present several properties of weak index pairs that will be used for the construction of the Conley index.
\begin{lm}\label{int3}
Let $S$ be an isolated invariant set, $N$ an isolating neighbourhood of $S$ and $P=(P_1,P_2)$ a weak index pair in $N$. If $M\subset N$ is an isolating neighbourhood for $S$ such that $P_1\setminus P_2\subset \Int M$, then $Q:=P\cap M$ is a weak index pair in $M$.
\end{lm}
\begin{proof}
By the property (a) of $P$ we have $F(Q_i)\cap N\subset F(P_i)\cap N\subset P_i$.
Thus $F(Q_i)\cap M\subset P_i\cap M=Q_i$ which shows that $Q$ satisfies (a).

Observe that $\bd _F(Q_1)\subset P_1\cap\cl (F(P_1)\setminus (P_1\cap M))\subset P_1\cap\cl (F(P_1)\setminus P_1)\cup P_1\cap\cl (F(P_1)\setminus M)$ which, according to (b) satisfied by $P$, yields
\begin{equation}\label{2}
\bd _F(Q_1)\subset P_2\cup P_1\cap\cl (F(P_1)\setminus M).
\end{equation}
For the indirect proof of (b) for $Q$ suppose the contrary, i.e. assume there exists a  $y\in \bd _F(Q_1)\setminus Q_2$. Then $y\in M$ and, by (\ref{2}), $y\in P_1\cap\cl (F(P_1)\setminus M)$. Thus, we have a sequence $\{y_n\}\subset F(P_1)\setminus M$ with $y=\lim_{n\to\infty}y_n$. However, $y\in P_1\setminus P_2\subset\Int M$ which means that $y_n\in \Int M$ for large enough $n\in\N$ and brings a contradiction.

Observe that $Q_1\setminus Q_2=(P_1\setminus P_2)\cap M$ which, along with $P_1\setminus P_2\subset\Int M$, implies $Q_1\setminus Q_2=P_1\setminus P_2$. Thus conditions (c) and (d) for $Q$ follow.
\qed\end{proof}
\begin{lm}\label{l_bd}
If $P$ is a weak index pair in $N$ then $\bd _F(P_1)\subset\bd N$.
\end{lm}
\begin{proof}
Suppose, for contradiction, that there exists a $y\in \bd _F(P_1)\setminus\bd N$ and consider
a sequence $\{y_n\}\subset F(P_1)\setminus P_1$ with $y=\lim_{n\to\infty}y_n$. Notice that $y\in\Int N$. Thus, for almost all $n\in\N$ we have $y_n\in N$, and property (a) of $P$ yields $y_n\in P_1$, a contradiction.
\qed\end{proof}
\begin{lm}\label{inter1}
If $P$ and $Q$ are weak index pairs in $N$ then so is $P\cap Q$.
\end{lm}
\begin{proof}
Applying property (a) of $P$ we can write $F(P_i\cap Q_i)\cap N\subset F(P_i)\cap N\subset P_i$. By the symmetry with respect to $P$ and $Q$ we also have $F(P_i\cap Q_i)\cap N\subset Q_i$. The above inclusions prove condition (a) for $P\cap Q$.

In order to verify that $P\cap Q$ satisfies (b) suppose the contrary and consider a $y\in \bd _F(P_1\cap Q_1)\setminus (P_2\cap Q_2)$. Then, $y\in P_1\cap Q_1$. Moreover, there exists a sequence $\{y_n\}\subset F(P_1\cap Q_1)\setminus (P_1\cap Q_1)$ such that $\lim_{n\to\infty}y_n=y$. Without loss of generality we may assume that either $\{y_n\}\cap P_1=\emptyset$ or $\{y_n\}\cap Q_1=\emptyset$. Since the above cases are analogous let us assume that $\{y_n\}\cap P_1=\emptyset$. Then, $y_n\in F(P_1)\setminus P_2$ and $y=\lim_{n\to\infty}y_n\in \cl (F(P_1)\setminus P_2)$. Since $y\in P_1$ it follows that $y\in \bd _F(P_1)$ which along with property (b) of $P$ results with $y\in P_2$. Consequently, $y\in Q_1\setminus Q_2\subset \Int N$ as $y\notin P_2\cap Q_2$ and $y\in Q_1$. However, by Lemma \ref{l_bd}, $y\in\bd N$, a contradiction.

By property (c) of $P$ and $Q$ we have $\Inv N\subset\Int (P_1\setminus P_2)\cap \Int (Q_1\setminus Q_2)=\Int (P_1\cap Q_1\setminus (P_2\cup Q_2))\subset \Int (P_1\cap Q_1\setminus (P_2\cap Q_2))$, hence $P\cap Q$ satisfies (c).

We also have (d). Indeed, $P_1\cap Q_1\setminus P_2\cap Q_2=(P_1\setminus P_2)\cap Q_1\cup (Q_1\setminus Q_2)\cap P_1\subset (P_1\setminus P_2)\cup (Q_1\setminus Q_2)\subset\Int N$.
\qed\end{proof}
\begin{lm}\label{inter2}
If $P\subset Q$ are weak index pairs in $N$ then so are $(P_1,P_1\cap Q_2)$ and $(P_1\cup Q_2,Q_2)$.
\end{lm}
\begin{proof}
Let us start with $(P_1,P_1\cap Q_2)$ and observe that $P_1\setminus (P_1\cap Q_2)\subset P_1\setminus P_2\subset\Int N$. Hence, we have (d).
Clearly $\bd _F(P_1)\subset P_2\subset P_1\cap Q_2$, thus (b) holds.
Verification of properties (a) and (c) is straightforward.

Property (d) of $(P_1\cup Q_2,Q_2)$ holds as we have $(P_1\cup Q_2)\setminus Q_2=P_1\setminus Q_2\subset Q_1\setminus Q_2\subset\Int N$. Verification of (a) and (c) is routine.

It remains to prove (b). For the indirect proof suppose that $y\in\bd _F(P_1\cup Q_2)\setminus Q_2$. There exist sequences $\{x_n\}\subset P_1\cup Q_2$ and $\{y_n\}\subset F(x_n)\setminus (P_1\cup Q_2)$ with $\lim_{n\to\infty}y_n=y$. Without loss of generality we may assume that either $\{x_n\}\subset P_1$ or $\{x_n\}\subset Q_2$. Suppose at first that $\{x_n\}\subset P_1$. Since $y_n\in F(P_1)\setminus P_1$ and $y\in P_1$ we have $y\in \bd _F(P_1)\subset P_2\subset Q_2$, a contradiction. If $\{x_n\}\subset Q_2$ then $y_n\in F(Q_2)\subset F(Q_1)$ and, by property (a) of $Q$, we infer that $y_n\notin Q_1$ as $y_n\notin Q_2$. Since $y\in P_1\subset Q_1$, we have $y\in \bd _F(Q_1)\subset Q_2$, a contradiction.
\qed\end{proof}
\begin{lm}\label{star}
If $A\subset B\subset N$ and $A$ is positively invariant in $N$, i.e. $F(A)\cap N\subset A$, 
then $\bd _F(A)\subset \bd _F(B)$.
\end{lm}
\begin{proof}
If $y\in F(A)\setminus A$, then $y\notin N$ which yields $y\notin B$. Clearly, $y\in F(B)$. Thus, $y\in F(B)\setminus B$. This shows that $F(A)\setminus A\subset F(B)\setminus B$. In consequence, $\bd _F(A)=\cl A\cap \cl (F(A)\setminus A)\subset\cl B\cap \cl (F(B)\setminus B)=\bd _F(B)$.
\qed\end{proof}
\begin{lm}\label{lem_G}
Let $P\subset Q$ be weak index pairs in $N$. Define a pair of sets $G(P,Q)$ by
$$
G_i(P,Q)=P_i\cup(F(Q_i)\cap N)\for i=1,2.
$$
If $P_1=Q_1$ or $P_2=Q_2$ then
\begin{itemize}
\item[(i)] $P_i=Q_i$ implies $G_i(P,Q)=P_i=Q_i$, for $i=1,2$,
\item[(ii)] $P\subset G(P,Q)\subset Q$,
\item[(iii)] $G(P,Q)$ is a weak index pair in $N$.
\item[(iv)] $F(Q_i)\cap N\subset G_i(P,Q)$, for $i=1,2$.
\end{itemize}
\end{lm}
\begin{proof}
Condition (i) follows from property (a) of weak index pairs. The first inclusion in (ii) is obvious. The second one follows from the first one and  property (a) of $Q$.
Condition (iv) is obvious.

We shall prove (iii). We start with property (d). If $P_1=Q_1$ then $G_1(P,Q)\setminus G_2(P,Q)\subset P_1\setminus P_2\subset\Int N$. If $P_2=Q_2$ then we have $G_1(P,Q)\setminus G_2(P,Q)\subset Q_1\cup F(Q_1)\cap N\setminus P_2\subset Q_1\setminus Q_2\subset\Int N$.

In order to verify (a) let $x\in G_i(P,Q)$ and $y\in F(x)\cap N$. If $x\in P_i$ then obviously $y\in G_i(P,Q)$. If $x\in F(Q_i)\cap N$ then $x\in Q_i$. Hence, $y\in F(Q_i)\cap N\subset G_i(P,Q)$.

By Lemma \ref{star} and (i) we have $\bd _FG_1(P,Q)\subset\bd _F(Q_1)$. Thus, if $P_1=Q_1$, then $\bd _FG_1(P,Q)\subset\bd _F(Q_1)=\bd _F(P_1)\subset P_2\subset G_2(P,Q)$. If $P_2=Q_2$, then $\bd _FG_1(P,Q)\subset\bd _F(Q_1)\subset Q_2=P_2\subset G_2(P,Q)$.

For (c), let us note that $\Inv N\subset\Int (P_1\setminus P_2)\cap\Int (Q_1\setminus Q_2)$, hence it is enough to show that $(P_1\setminus P_2)\cap(Q_1\setminus Q_2)\subset G_1(P,Q)\setminus G_2(P,Q)$. Indeed, if $y\in (P_1\setminus P_2)\cap(Q_1\setminus Q_2)$ then $y\in G_1(P,Q)$ and the supposition that $y\in G_2(P,Q)$ implies $y\in F(Q_2)\cap N\subset Q_2$, a contradiction.
\qed\end{proof}
\begin{lm}\label{seq}
Let $P\subset Q$ be weak index pairs in $N$ such that $P_1=Q_1$ or $P_2=Q_2$. Then, there exists a sequence of weak index pairs such that
$$
P=Q^n\subset Q^{n-1}\subset\cdots\subset Q^1\subset Q^0=Q
$$
with the following properties
\begin{itemize}
\item[(i)] $P_i=Q_i$ implies $Q^k _i=P_i=Q_i$ for all $k=1,2,\dots, n-1$, $i=1,2$,
\item[(ii)] $F(Q^k _i)\cap N\subset Q^{k+1} _i$ for $k=0,1,\dots, n-1$, $i=1,2$.
\end{itemize}
\end{lm}
\begin{proof}
Let us define $Q^k$ by the recurrence formula $Q^0=Q$, $Q^{k+1}=G(P,Q^k)$ for $k\in\Z ^+$. By Lemma \ref{lem_G} and induction on $k$, $\{Q^k\}$ is a decreasing sequence of weak index pairs satisfying (i) and (ii) for all $k\in\Z^+$. It remains to show that for some $n$ we have $Q^n=P$. To the contrary suppose that the inclusion $P\subset Q^k$ is strict for all $k$. Fix $k\in\Z^+$ and $i\in\{1,2\}$ such that $Q_i ^k\setminus P_i\neq \emptyset$, and observe that
\begin{equation}\label{sol}
\mbox{there exists a solution }\sigma _k:[0,k]\to Q_i\setminus P_i.
\end{equation}
Indeed, choose $\sigma_k(k)\in Q_i ^k\setminus P_i\subset Q_i\setminus P_i$. Then $\sigma_k(k)\in F(Q_i ^{k-1}\cap N)$, hence there exists $\sigma_k(k-1)\in Q_i ^{k-1}$ with $\sigma_k(k)\in F(\sigma_k(k-1))$. By property (a) of $P$ we may assume that $\sigma_k(k-1)\notin P_i$. Hence, $\sigma_k(k-1)\in Q_i ^{k-1}\setminus P_i\subset Q_i\setminus P_i$. By the reverse recurrence we are done.

Since $k$ was arbitrarily chosen, by (\ref{sol}) and Lemma \ref{lem_8}, we have
\begin{equation}\label{eq5}
\Inv (Q_i\setminus\Int P_i)\neq 0.
\end{equation}
For $i=1$ we get $\emptyset\neq\Inv (Q_1\setminus\Int P_1)\subset Q_1\setminus\Int P_1$. However, by property (c) of $P$, $\Inv (Q_i\setminus\Int P_i)\subset\Inv N\subset\Int (P_1\setminus P_2)\subset\Int P_1$, a contradiction.

If $i=2$ then $\Inv Q_2\neq\emptyset$, because by (\ref{eq5}), $\emptyset\neq\Inv (Q_2\setminus\Int P_2)\subset\Inv Q_2$. On the other hand, by property (c) of $Q$, $\Inv Q_2\subset\Inv N\subset\Int (Q_1\setminus Q_2)\subset Q_1\setminus Q_2$
which means that $\Inv Q_2=\emptyset$, a contradiction.
\qed\end{proof}
\section{Definition of the Conley index}
Since now on we will require that the generator $F:X\to\mathcal{P}(X)$ of the dmds, restricted to appropriate pairs of sets, induces a homomorphism in cohomology. Therefore, we assume that $F$ is {\em determined by a given morphism}. For the details concerning this concept we refer to \cite{G83}, \cite{G76} and \cite{M90}. Let us remind that, in particular, any single-valued continuous map as well as any composition of acyclic maps (i.e. usc maps with compact acyclic values) belongs to this class.

For a weak index pair $P$ in an isolating neighbourhood $N$ we let
$$
T(P):=T_N(P):=(P_1\cup (X\setminus \Int N), P_2\cup (X\setminus \Int N)).
$$
\begin{lm}\label{l1}
If $P$ is a weak index pair for $F$ in $N$ then
\begin{itemize}
\item[(i)] $F(P)\subset T(P)$,
\item[(ii)] the inclusion $i_{P,T(P)}:P\to T(P)$ induces an isomorphism in the Alexander-Spanier cohomology.
\end{itemize}
\end{lm}
\begin{proof}
Condition (i) follows from property (a) of $P$. Since we have $T_1(P)\setminus T_2(P)=(P_1\setminus P_2)\cap N=P_1\setminus P_2$, and $P_1\setminus P_2\subset\Int N$, condition (ii) is a consequence of the strong excision property (cf. \cite{Sp1966}).
\qed\end{proof}
Let $F_{P,T(P)}(x):=F(x)$ for $x\in P$. By Lemma \ref{l1} (i) such a restriction of $F$ is a map of pairs $F_{P,T(P)}:P\to \mathcal{P}(T(P))$. 
Put $i_P:=i_{P,T(P)}$.
By Lemma \ref{l1} (ii) $H^*(i_P)$ is an isomorphism. Thus, we can pose the following definition.
\begin{df}
{\rm
The endomorphism $H^*(F_{P,T(P)})\circ H^*(i_P)^{-1}$ of $H^*(P)$ will be called the {\em index map} associated with the index pair $P$ and denoted by $I_P$.
}
\end{df}
Applying the Leray functor $L$ to $(H^*(P),I_P)$ we obtain a graded module over $\Z$ and its endomorphism, which is called the Leray reduction of the Alexander-Spanier cohomology of $P$. For the details we refer to \cite{M90}.
\begin{df}
\label{def:con}
{\rm
The module $L(H^*(P),I_P)$ will be called the {\em cohomological Conley index of $S$} and denoted by $C(S,F)$, or simply  by $C(S)$ if $F$ is clear from the context.}
\end{df}
To have the Conley index well defined, the following theorem, which is an analogue of \cite[Theorem 3.2]{KM95}, is necessary.
\begin{thm}\label{th_ind}
Let $S$ be an isolated invariant set. Then $C(S,F)$ is independent of the choice of an isolating neighbourhood $N$ for $S$ and of a weak index pair $P$ in $N$.
\end{thm}
\begin{proof} Let $M$, $N$ be two isolating neighbourhoods of $S$ and $P$, $Q$ weak index pairs in $N$ and $M$, respectively. We need to prove that $L(H^*(P),I_P)=L(H^*(Q),I_Q)$. The proof runs in five steps, similarly as the proof of \cite[Theorem 3.2]{KM95}.

{\em Step 1.} Let us consider the following special case
\begin{itemize}
\item[(i)] $M=N$,
\item[(ii)] $P\subset Q$,
\item[(iii)] $P_1=Q_1$ or $P_2=Q_2$,
\item[(iv)] $F(Q)\subset T_N(P)$.
\end{itemize}
By (iv) we may treat $F$ as a map of pairs $F_{Q,T(P)}:Q\to\mathcal{P}(T_N(P))$. Let $I_{Q,P}:=H^*(F_{Q,T(P)})\circ H^*(i_P)^{-1}$ be the induced homomorphism. We have the following commutative diagram
\begin{center}
\begin{tabular}{ccc}
$H^*(P)$&$\xleftarrow{\hspace{5mm}I_P\hspace{5mm}}$&
$H^*(P)$\\[0.4cm]
\small{$H^*(j)$}$\left\uparrow\rule{0cm}{0.8cm}\right.$&$\swarrow$\small{$I_{Q,P}$}&$\left\uparrow\rule{0cm}{0.8cm}\right. $\small{$H^*(j)$}\\[8mm]
$H^*(Q)$&$\xleftarrow{\hspace{5mm}I_Q\hspace{5mm}}$&$
H^*(Q)$
\end{tabular}
\end{center}
in which $j:P\to Q$ is the inclusion. This shows that $(H^*(P),I_P)$ and $(H^*(Q),I_Q)$ are linked in the sense of \cite{M90} hence $LH^*(j):L(H^*(Q),I_Q)\to L(H^*(P),I_P)$ is an isomorphism.

{\em Step 2.} We lift the assumption (iv). Let $Q^k$, $k=1,2,\dots, n$ be the sequence of weak index pairs, existence of which follows from Lemma \ref{seq}. Then each pair $(Q^{k+1}$, $Q^k)$ of weak index pairs satisfies assumptions (ii) -- (iv), hence by {\em Step 1} their corresponding Leray reductions are isomorphic. Since $Q_0=Q$ and $Q_n=P$, the conclusion follows.

{\em Step 3}. We lift the assumption (iii). Let $R_1:=(P_1\cup Q_2)$, $R_2:=P_1\cap Q_2$. By Lemma \ref{inter2} $(P_1,R_2)$ and  $(R_1,Q_2)$ are weak index pairs. Consider the following commutative diagram of inclusions
\begin{center}
\begin{tabular}{ccc}
$(P_1,R_2)$&$\xrightarrow{\hspace{5mm}j_2\hspace{5mm}}$&
$(R_1,Q_2)$\\[0.4cm]
\small{$j_1$}$\left\uparrow\rule{0cm}{0.6cm}\right.$&&$\left\uparrow\rule{0cm}{0.6cm}\right. $\small{$j_3$}\\[8mm]
$(P_1,P_2)$&$\xrightarrow{\hspace{5mm}j_4\hspace{5mm}}$&$
(Q_1,Q_2)$
\end{tabular}
\end{center}
Obviously pairs $(P_1,P_2)$, $(P_1,R_2)$ and $(R_1,Q_2)$, $(Q_1,Q_2)$ satisfy assumptions (ii) and (iii), therefore by {\em Step 2} inclusions $j_1$ and $j_3$ induce isomorphisms. Since $P_1\setminus R_2=P_1\setminus Q_2=R_1\setminus Q_2$, by the strong excision property inclusion $j_2$ induces an isomorphism.

{\em Step 4}. We assume only (i). By Lemma \ref{inter1} $P\cap Q$ is a weak index pair, hence the conclusion follows from {\em Step 3} applied to pairs $P\cap Q\subset P$ and $P\cap Q\subset Q$.

{\em Step 5}. If $M\neq N$ one may assume that $M\subset N$ since otherwise $M\cap N$ can be considered as an isolating neighbourhood of $S$. By {\em Step 4} it suffices to show the existence of weak index pairs $P$ and $Q$ in $N$ and $M$ respectively, and such that $L(H^*(P),I_P)$ and $L(H^*(Q),I_Q)$ are isomorphic.

By Proposition \ref{wip} there exists a weak index pair $P$ in $N$ such that $P_1\setminus P_2\subset \Int M$. By Lemma \ref{int3}, $Q:=P\cap M$ is a weak index pair in $M$. Moreover $Q_1\setminus Q_2=(P_1\setminus P_2)\cap M=P_1\setminus P_2$ hence the inclusion $Q\subset P$ induces an isomorphism in cohomology, by the strong excision property.
\qed\end{proof}

We finish this section with the following theorem showing that the definition of the Conley index for dmds proposed
in this paper generalizes earlier definitions.
\begin{thm}
\label{thm:gen}
Assume $S$ is a strongly isolated invariant set of a dmds $F$. Then,
$S$ is an isolated invariant set in the sense of Definition~\ref{def:iis}.
Moreover, the Conley index of $S$ in the sense of Definition~\ref{def:con}
coincides with the Conley index of $S$ in the sense of \cite{KM95,S06}.
\end{thm}
\begin{proof}
The theorem is an immediate consequence of Propositions~\ref{prop:siis-iis},~\ref{prop_strong_in},
Theorem 2.6 in \cite{KM95}, Proposition~\ref{prop:ip-wip}
and Theorem~\ref{th_ind}.
\qed\end{proof}

\section{Construction of the dual discrete multivalued dynamical system admitting index pairs}
From now on we assume that $X$ is a locally compact normal space and $P$ is a weak index pair in $N$.

We select compact, disjoint sets $C,D\subset P_1\cup (X\setminus \Int N)$ such that $$\cl (P_1\setminus P_2)\subset C$$ and $$X\setminus \Int N\subset D.$$ By Urysohn's lemma we can choose a continuous function $\alpha :X\to [0,1]=:I$ such that $\alpha |_C=0$ and $\alpha |_D=1$.
Let
$$
\bar{X}:=\bar{X}(P):=(P_1\setminus P_2)\times \{0\}\cup
(P_2\cup (X\setminus \Int N))\times I\cup X\times \{1\}
$$
with the Tichonov topology.
Consider $$\mu:\bar{X}\ni (x,t)\mapsto t+(1-t)\alpha (x)\in I.$$
The following properties of $\mu$ are straightforward.
\begin{prop}\label{prop_mu}
For any $(x,t)\in\bar{X}$ we have
\begin{enumerate}
\item $\mu (x,0)=\alpha (x)$,
\item $x\in C\Rightarrow\mu(x,t)=t$,
\item $x\in D\vee t=1\Rightarrow\mu(x,t)=1$,
\item $\mu$ is nondecreasing with respect to the second variable.
\end{enumerate}
\end{prop}
\begin{prop}\label{Fbar}
We have a well defined usc, acyclic valued map
$$
\bar{F}:\bar{X}\ni(x,t)\mapsto F(x)\times\{\mu(x,t)\}\subset \bar{X}.
$$
\end{prop}
\begin{proof}
We shall verify that $F(x)\times\{\mu(x,t)\}\subset \bar{X}$ for any $(x,t)\in \bar{X}$. According to Proposition \ref{prop_mu} (3) the conclusion is straightforward whenever $x\in X\setminus \Int N$ or $t=1$.
By Lemma \ref{l1} (i) we have $\bar{F}(x,t)\subset (P_2\cup(X\setminus \Int N))\times I\subset \bar{X}$ for $x\in P_2$ and any $t\in I$. It remains to consider $x\in P_1\setminus P_2$ and $t=0$. Then, by Lemma \ref{l1} (i) and Proposition \ref{prop_mu} (2),
$\bar{F}(x,0)\subset 
(P_1\cup(X\setminus \Int N))\times\{0\}\subset\bar{X}$.

Upper semicontinuity of $\bar{F}$ follows from the upper semicontinuity of $F$ and the continuity of $\mu$.

Since for any $(x,t)\in\bar{X}$ the set $\bar{F}(x,t)$ is homeomorphic to $F(x)$, acyclicity of values is obvious.
\qed\end{proof}
Consider the homeomorphism (onto its image)
$$
\iota:P_1\cup(X\setminus \Int N)\ni x\mapsto (x,0)\in\bar{X}.
$$
From the lines of the proof of Proposition \ref{Fbar} the following corollary follows.
\begin{cor}\label{c1}
We have
\begin{itemize}
\item[(i)] $\bar{F}(\iota(\cl (P_1\setminus P_2)))\subset \iota(P_1\cup(X\setminus \Int N))$,
\item[(ii)] $\bar{F}((X\setminus \Int N)\times I)\subset X\times \{1\}$.
\end{itemize}
\end{cor}
Let
$$
\eta:(P_1\setminus P_2)\times\{0\}\cup (P_2\cup(X\setminus \Int N))\times I \ni (x,t)\mapsto (x,0)\in\bar{X}.
$$
\begin{lm}\label{lp3}
We have
\begin{itemize}
\item[(i)] $\iota(F(x))=\eta (\bar{F}(\iota (x)))$ for $x\in P_1$,
\item[(ii)] $\iota(F(x))=\bar{F}(\iota (x))$ for $x\in \cl (P_1\setminus P_2)$.
\end{itemize}
\end{lm}
\begin{proof}
In order to verify (i) consider $x\in P_1$. Then $\eta (\bar{F}(\iota (x)))=\eta (\bar{F}(x,0))=F(x)\times\{0\}=\iota(F(x))$. Condition (ii) is a consequence of (i) and Corollary \ref{c1} (i).
\qed\end{proof}
\begin{prop}\label{prop3}
Let $S$ be an isolated invariant set for $F$ in its isolating neighbourhood $N$, and let $P$ be a weak index pair. Assume that $\bar{X}$ and $\bar{F}$ are defined as above. Then $\bar{S}:=\iota (S)$ is an isolated invariant set for $\bar{F}$, $\bar{N}:=\iota ((P_1\cup(X\setminus \Int N))\cap N)$ is its isolating neighbourhood and $\bar{P}:=\iota (P)$ is a weak index pair for $\bar{S}$ in $\bar{N}$.
\end{prop}
\begin{proof}
The set $\bar{N}$ is compact as a homeomorphic image of a compact set. By property (c) of $P$ 
we have $\bar{S}\subset \iota(\Int (P_1\setminus P_2))=
\Int (\iota(P_1\setminus P_2))=
\Int (\bar{P_1}\setminus \bar{P_2})\subset\Int \bar{N}$. Thus $\bar{S}\subset\Int \bar{N}$ and $\bar{P}$ satisfies (c).

We shall verify that $\Inv \bar{N}=\bar{S}$. Let $(x,0)\in\Inv \bar{N}$ and let $\bar{\sigma}:\Z\to\bar{N}$ be a solution for $\bar{F}$ with $\bar{\sigma}(0)=(x,0)$. By Proposition \ref{prop_mu} (3), $\bar{\sigma}(k)\in \bar{P_1}$ for $k\in\Z$. 
Define $\sigma(k):\Z\to P_1$
by $\sigma(k):=\iota ^{-1}(\bar{\sigma}(k))$ for $k\in\Z$. Then $\sigma(0)=x$ and, by Lemma \ref{lp3} (i) for any $k\in\Z$, $\sigma(k+1)=\iota ^{-1}(\bar{\sigma}(k+1))\in \iota ^{-1}(\bar{F}(\bar{\sigma}(k))=\iota ^{-1}(\bar{F}(\iota(\sigma(k)))=F(\sigma(k))$, which means that $\sigma$ is a solution for $F$ through $x$ in $N$. Therefore $(x,0)\in\bar{S}$. We have proved that $\Inv \bar{N}\subset\bar{S}$. In order to prove the opposite inclusion consider $(x,0)\in \bar{S}$. Since $M:=\cl (P_1\setminus P_2)$ is an isolating neighbourhood for $S$, there exists $\sigma:\Z\to M$, a solution for $F$ through $x$. One can check that then $\iota\circ\sigma$ is a solution for $\bar{F}$ through $(x,0)$ in $\iota(M)\subset\bar{N}$, hence, $(x,0)\in \Inv \bar{N}$.

Now we will show that $\bar{P}$ is a weak index pair for $\bar{S}$ in $\bar{N}$. Note that property (c) of $\bar{P}$ has already been verified. Moreover, it is clear that $\bar{P_2}\subset\bar{P_1}$ are compact subsets of $\bar{N}$ with $\bar{P_1}\setminus\bar{P_2}\subset\Int \bar{N}$, hence (d) is verified. In order to verify condition (a) for $\bar{P}$ observe that
$$
\begin{array}{rcl}
\bar{F}(\bar{P_i})\cap\bar{N}&=&
\bar{F}(\iota(P_i))\cap\iota(N\cap(P_1\cup(X\setminus\Int N)))\\
&\subset&
\eta(\bar{F}(\iota(P_i)))\cap\iota(N\cap(P_1\cup(X\setminus\Int N)))\\
&=&\iota(F(P_i)\cap N\cap(P_1\cup(X\setminus\Int N))),
\end{array}
$$
where the last equality follows from Lemma \ref{lp3} (i). By the above inclusion and  property (a) of $P$, we have
$
\bar{F}(\bar{P_i})\cap\bar{N}\subset
\iota(P_i\cap(P_1\cup(X\setminus\Int N)))
=\bar{P_i}.
$

It remains to prove condition (b). We have
$$
\begin{array}{rcl}
\bd _{\bar{F}}(\bar{P_1})&=&\iota (P_1)\cap\cl (\bar{F}(\iota(P_1))\setminus\iota(P_1))\\
&=&\iota (P_1)\cap\cl (\bar{F}(\iota(P_1\setminus P_2))\cup\bar{F}(\iota(P_2))\setminus\iota(P_1))
\end{array}
$$
and consequently, using Lemma \ref{lp3} (ii), we get
$$
\begin{array}{rcl}
\bd _{\bar{F}}(\bar{P_1})&=&\iota (P_1)\cap\cl [\iota(F(P_1\setminus P_2))\cup\bar{F}(\iota(P_2))\setminus\iota(P_1)]\\
&\subset&\iota (P_1)\cap[\cl (\iota(F(P_1\setminus P_2)\setminus P_1))\cup\cl (\bar{F}(\iota(P_2)))]\\
&=&\iota (P_1\cap\cl (F(P_1\setminus P_2)\setminus P_1))\cup(\iota (P_1)\cap\cl (\bar{F}(\iota(P_2))))\\
&\subset&\iota(P_2)\cup(\iota (P_1)\cap\cl (\bar{F}(\iota(P_2)))),
\end{array}
$$
where the last inclusion follows from property (b) of $P$. By Lemma \ref{l1} (i) we have $\bar{F}(\iota(P_2))\subset (P_2\cup(X\setminus \Int N))\times I$, hence
$$
\begin{array}{rcl}
\iota (P_1)\cap\cl (\bar{F}(\iota(P_2)))&\subset&
\iota (P_1)\cap P_2\times I\cup \iota (P_1)\cap (X\setminus \Int N))\times I\\
&\subset&\iota (P_2)\cup \iota (P_1)\cap (X\setminus \Int N))\times I\\
\end{array}
$$
and, finally, $\iota (P_1)\cap (X\setminus \Int N)\times I\subset\bar{P_2}$ as $P_1\setminus P_2\subset\Int N$.
\qed\end{proof}
\begin{thm}\label{csbar}
Let $S$ be an isolated invariant set for $F$ in its isolating neighbourhood $N$, and let $P$ be a weak index pair. Assume that $\bar{X}$, $\bar{F}$ and $\bar{S}$ are defined as in Proposition \ref{prop3}. Then $C(S,F)=C(\bar{S},\bar{F})$.
\end{thm}
\begin{proof}
By Proposition \ref{prop3}, $\bar{S}$ is an isolated invariant set for $\bar{F}$. Consider a weak index pair $\bar{P}$ for $\bar{S}$ in $\bar{N}$, as defined in Proposition \ref{prop3}.
By Theorem \ref{th_ind} it suffices to show that $L(H^*(P),I_P)$ and $L(H^*(\bar{P}),I_{\bar{P}})$ are isomorphic.

By Lemma \ref{l1} (i), $F$ is a map of pairs $(P_1,P_2)$ and $(T_{N,1}(P),T_{N,2}(P))$. By similar reasons $\bar{F}$ maps $(\bar{P_1},\bar{P_2})$ into $(T_{\bar{N},1}(\bar{P}),T_{\bar{N},2}(\bar{P}))$.
We have the commutative diagram
\begin{center}
\begin{tabular}{ccccc}
$(P_1,P_2)$&$\stackrel{F}{\longrightarrow }$&
$(T_{N,1}(P),T_{N,2}(P))$
&$\stackrel{j_1}{\longleftarrow}$&$(P_1,P_2)$\\[4mm]


$\left\uparrow\rule{0cm}{0.5cm}\right.$\small{$p_1$}&&$\left\uparrow\rule{0cm}{0.5cm}\right.$\small{$p_2$}&&$\left\uparrow\rule{0cm}{0.5cm}\right.$\small{$p_1$}\\[4mm]

$(\bar{P_1},\bar{P_2})$&$\stackrel{\bar{F}}{\longrightarrow}$&
$(T_{\bar{N},1}(\bar{P}),T_{\bar{N},2}(\bar{P}))$
&$\stackrel{j_2}{\longleftarrow}$&$(\bar{P_1},\bar{P_2})$
\end{tabular}
\end{center}
in which $j_1, j_2$ are inclusions and
$$
\begin{array}{l}
p_1:(\bar{P_1},\bar{P_2})\ni(x,0)\mapsto x\in (P_1,P_2),\\
p_2:(T_{\bar{N},1}(\bar{P}),T_{\bar{N},2}(\bar{P}))\ni (x,t)\mapsto x\in (T_{N,1}(P),T_{N,2}(P))
\end{array}
$$
are projections. 
By Lemma \ref{l1} (ii)
inclusion $j_1$ induces an isomorphism in cohomology, and so does $j_2$. Since, additionally, projection $p_1$ induces an isomorphism, it follows that $p_2$ induces an isomorphism. Eventually we infer that $I_P$ and $I_{\bar{P}}$ are conjugate.
\qed\end{proof}
\begin{thm}
Under the regime of Theorem \ref{csbar}, $\bar{S}$ is a strongly isolated invariant set for $\bar{F}$.
\end{thm}
\begin{proof}
We need to find a strongly isolating neighbourhood $M$ for $\bar{S}$.
Let $N$, $P$, $\bar{N}$ and $\bar{P}$ be as in Proposition \ref{prop3}. Define $M:=(P_1\setminus P_2)\times\{0\}\cup (P_2\cup (X\setminus \Int N))\times [0,\frac{1}{2}]$.
Clearly $\bar{S}\subset M$ are compact subsets of $\bar{X}$. By Proposition \ref{prop3}, $\bar{S}\subset \Int \bar{N}\subset \Int M$ and $\bar{S}=\Inv \bar{N}\subset \Inv M$. We shall prove that $\Inv M\subset\bar{S}$. By Corollary \ref{c1} (ii) we have $\Inv M\cap (X\setminus\Int N)\times I=\emptyset$. Suppose that $(x,t)\in \Inv M\cap [((P_1\setminus P_2)\times\{0\}\cup P_2\times [0,\frac{1}{2}])\setminus \bar{S}]$. Let $\bar{\sigma}:\Z\to M$ be a solution for $\bar{F}$ with $\bar{\sigma}(0)=(x,t)$. According to Corollary \ref{c1} (ii) we may assume that $\bar{\sigma}(k)\in M\setminus (X\setminus \Int N)\times[0,\frac{1}{2}]$ for $k\in\Z$. Then $\eta\circ \bar{\sigma}:\Z\to \bar{N}$ is a solution for $\bar{F}$ through $(x,0)$ hence, by Proposition \ref{prop3}, $(x,0)\in\bar{S}$, a contradiction.

By Corollary \ref{c1} (i), $\bar{F}(\bar{S})\subset (P_1\cup(X\setminus \Int N))\times\{0\}\subset\Int M$, which along with $\bar{S}\subset\Int M$ yields $\bar{S}\cup\bar{F}(\bar{S})\subset\Int M$.
\qed\end{proof}
\section{Examples}\label{sec:ex}
\begin{ex}
{\em Let $f$ be given by \eqref{eq:f} and let $F$ be defined as in Section \ref{sec:m-ex} (see Figure \ref{fig:z2on16samples}).
As in Example \ref{e1} consider $0$, a hyperbolic fixed point of $f$ and an isolated invariant set $S:=[\frac{31}{32},\frac{1}{32}]\in\cA$ for $F$ containing $0$.

We are going to show that the Conley index of $S$ for $F$ is the same as the Conley index of $\{0\}$ for $f$.

Observe that the set $N:=[\frac{15}{16},\frac{1}{16}]$ is an isolating neighbourhood
for $F$ and $S$. The sets $P_1:=N$ and $P_2:=\{\frac{15}{16}\}\cup \{\frac{1}{16}\}$ give raise to a weak index pair $P=(P_1,P_2)$ for $F$ in $N$. Then $H_k ^*(P)$ has one generator for $k=1$ and is trivial for other $k$. An easy computation shows that
$$
C_k(S,F)=\left\{\begin{array}{rl}
(\Z,\id)&\mbox{for } k=1\\[1ex]
0&\mbox{otherwise}.
\end{array}
\right.
$$
Let us note that $N$ is also an isolating neighbourhood for $\{0\}$ and $f$, and $P$ is a weak index pair for $f$ in $N$. It can be easily verified that
$$
C_k(\{0\},f)=\left\{\begin{array}{rl}
(\Z,\id)&\mbox{for } k=1\\[1ex]
0&\mbox{otherwise}.
\end{array}
\right.
$$
}
\end{ex}
\begin{ex}
{\em Let $f$ and $F$ be the same as above. In Example \ref{e2} we have observed that $S':=\{\frac{1}{3},\frac{2}{3}\}$ is a hyperbolic periodic trajectory of $f$. Moreover, $S:=[\frac{9}{32},\frac{13}{32}]\cup [\frac{19}{32},\frac{23}{32}]$ is a cover of $S'$ by elements of the grid $\cA$ and each point of $S$ belongs to a $2$-periodic trajectory of $F$ in $S$. The set $N:=[\frac{17}{64},\frac{27}{64}]\cup [\frac{37}{64},\frac{47}{64}]$ is an isolating neighbourhood for $S$ and $F$.

In order to compute $C(S,F)$, define $P_1:=N$ and $P_2:= \{\frac{17}{64},\frac{27}{64},\frac{37}{64},\frac{47}{64}\}$. It is straightforward to observe that $P=(P_1,P_2)$ is a weak index pair for $F$ in $N$. Then $H_k ^*(P)$ is trivial for $k\neq 1$ and it has 2 generators for $k=1$.
The index map is an isomorphism
$$
I_P=\left[\begin{array}{rr}
0&1\\
1&0
\end{array}
\right].
$$
We have
$$
C_k(S,F)=\left\{\begin{array}{rl}
(\Z ^2,\tau)&\mbox{for } k=1\\[1ex]
0&\mbox{otherwise,}
\end{array}
\right.
$$
where $\tau$ is a transposition $\tau:\Z ^2\ni(x,y)\mapsto (y,x)\in \Z^2.$

It is easy to verify that $N$ is an isolating neighbourhood for $S'$ and $f$, and $P$ is a weak index pair for $f$ in $N$. Again, $H_k ^*(P)$ is trivial for $k\neq 1$ and has two generators for $k=1$. The index map has the form
$$
I_P=\left[\begin{array}{rr}
0&1\\
1&0
\end{array}
\right]
$$
and
$$
C_k(S',f)=\left\{\begin{array}{rl}
(\Z ^2,\tau)&\mbox{for } k=1\\[1ex]
0&\mbox{otherwise.}
\end{array}
\right.
$$
Thus, the Conley index of $S$ for $F$ is the same as the Conley index
of $\{\frac{1}{3},\frac{2}{3}\}$ for $f$.
}
\end{ex}

\end{document}

%% file: MM_MathDefs.tex
\def\refeq#1{\if\workingver y(\ref{#1})-[[#1]]\else(\ref{#1})\fi}
\def\refth#1{\if\workingver y\ref{#1}-[[#1]]\else\ref{#1}\fi}
\def\mylabel#1{\if\workingver y\label{#1}{\bf\ \ [[#1]]\ \ }\else\label{#1}\fi}
\def\mybibitem#1{\if\workingver y\bibitem{#1}{\bf\ \ [[#1]]\ \
}\else\bibitem{#1}\fi}

\newfont{\msam}{msam10}
\newfont{\msbm}{msbm10}

\def\articletheorems{
\newtheorem{thm}{Theorem}[section]

\newtheorem{cor}[thm]{Corollary}
\newtheorem{prop}[thm]{Proposition}
\newtheorem{ex}[thm]{Example}
\newtheorem{algo}{Algorithm}[section] 
\newtheorem{alg}[thm]{Algorithm}  

}

\def\cA{\text{$\mathcal A$}}

\def\cG{\text{$\mathcal G$}}

\def\cK{\text{$\mathcal K$}}

\newcommand{\id}{\operatorname{id}}
\newcommand{\cl}{\operatorname{cl}}

\newcommand{\Int}{\operatorname{int}}

\newcommand{\bd}{\operatorname{bd}}

\newcommand{\diam}{\operatorname{diam}}

\newcommand{\dist}{\mbox{$\operatorname{dist}$}}

\renewcommand{\emptyset}{\varnothing}

\newcommand{\Inv}{\operatorname{Inv}}

\def\begeq#1{\begin{equation}\mylabel{#1}}
\def\endeq{\end{equation}}

\def\mathobj#1{\mbox{$#1$}}

\def\RR{\mathobj{\mathbb{R}}}

\def\ZZ{\mathobj{\mathbb{Z}}}

\def\setof#1{\mbox{$\{\,#1\,\}$}}

\def\0#1{\hbox{\kern25pt}$ #1 $\\}
\def\1#1{\hbox{\kern40pt}$ #1 $\\}
\def\2#1{\hbox{\kern55pt}$ #1 $\\}
\def\3#1{\hbox{\kern70pt}$ #1 $\\}

\newcounter{li}

\def\begalg#1{\begin{algo}\mylabel{#1}\normalshape:\small\baselineskip 10pt\\}
\def\endalg{\end{algo}}

\def\Figures(include=#1,cat=#2){
  \renewcommand{\textfraction}{.20}
  \renewcommand{\topfraction}{.80}
  \renewcommand{\bottomfraction}{.80}
  \renewcommand{\floatpagefraction}{.80}
  \newcount\figcount
  \figcount=0
  \let\includefigures=#1
  \def\figcat{#2}
}

\def\FigureFromFile[#1][#2](#3)#4
{
  \begin{figure}[htbp]
     \global\advance\figcount by 1
     \if\includefigures y\special{anisoscale #1.wmf, \the\hsize #2}\fi
     \vspace{#2}
     \caption{#4}
     \mylabel{#3}
   \end{figure}
}

\def\FigureFromFileTwoD[#1][#2,#3](#4)#5
{
  \begin{figure}[htbp]
     \global\advance\figcount by 1
     \if\includefigures y\special{anisoscale #1.wmf, #2 #3}\fi
     \vspace{#2}
     \caption{#5}
     \mylabel{#4}
   \end{figure}
}

\def\FigureF<#1>[#2](#3)#4
{
  \begin{figure}[htbp]
     \global\advance\figcount by 1
     \if\includefigures y\special{anisoscale \figcat/fig\number\figcount.wmf,
       \the\hsize #2}
     \fi
     \if\includefigures p
       \leavevmode
       \epsfxsize=\hsize
       \epsffile{#1}
     \fi
     \if\includefigures y
          \vspace{#2}
     \fi
     \caption{#4}
     \mylabel{#3}
   \end{figure}
}

\def\Figure[#1](#2)#3
{
  \begin{figure}[htbp]
     \global\advance\figcount by 1
     \if\includefigures y\special{anisoscale \figcat/fig\number\figcount.wmf,
       \the\hsize #1}
     \fi
     \if\includefigures p
       \leavevmode
       \epsfxsize=\hsize
       \epsffile{fig\number\figcount.eps}
     \fi
     \if\includefigures y
          \vspace{#1}
     \fi
     \caption{#3}
     \mylabel{#2}
   \end{figure}
}